\documentclass[11pt,a4paper,headinclude,footinclude,fleqn]{article}                               
\usepackage[T1]{fontenc}                   
\usepackage[utf8]{inputenc}                 
\usepackage[english]{babel}       
\usepackage{graphicx}                      % 
\usepackage[font=small]{quoting}            %  
\usepackage{caption}        
\usepackage{mathrsfs}          %                  
\usepackage[top=.8in,bottom=1in,left=1in,right=1in]{geometry}
\usepackage[english]{babel}
\usepackage{stmaryrd}
\usepackage{verbatim}
\usepackage{color}
\usepackage{picture}
\usepackage{amsmath,amsthm,amsfonts,amssymb}
\usepackage{comment}
\usepackage{mathtools}
\usepackage{pgf,tikz}
\usetikzlibrary{arrows}

\usepackage{titling} 
\newtheorem{thm}{Theorem}[section]
\newtheorem{lem}[thm]{Lemma}

\newtheorem{prop}[thm]{Proposition}

\theoremstyle{definition}
\newtheorem{rem}[thm]{Remark}

\theoremstyle{remark}

\newcommand{\ds}{\displaystyle}

\newcommand{\R}{\mathbb{R}}

\newcommand{\Rn}{\mathbb R^n}

\newcommand{\de}{\partial}

\DeclareMathOperator{\dive}{div}
\DeclareMathOperator{\Qp}{\mathcal Q_{p}}

{\left\{\begin{array}{@{}l@{}}}{\end{array}\right.}
\patchcmd{\abstract}{\scshape\abstractname}{\textbf{\abstractname}}{}{}
\makeatletter %note a di pagina senza numero 1
\def\@makefnmark{} %note a di pagina senza numero 2
\makeatother %note a di pagina senza numero 3
\title{ 
Sharp estimates on the first Dirichlet eigenvalue of nonlinear elliptic operators via maximum principle
}
\author{Francesco Della Pietra%, %
\thanks{Universit\`a degli studi di Napoli Federico II, Dipartimento di Matematica e Applicazioni ``R. Caccioppoli'', Via Cintia, Monte S. Angelo - 80126 Napoli, Italia. Email: f.dellapietra@unina.it} 
{, }
Giuseppina di Blasio%, %
\thanks{Universit\`a degli Studi della Campania ``Luigi Vanvitelli'', viale Lincoln, 5 81100 Caserta, Italy. Email:
giuseppina.diblasio@unicampania.it}
{, } 
Nunzia Gavitone %
\thanks{Universit\`a degli studi di Napoli Federico II, Dipartimento di Matematica e Applicazioni ``R. Caccioppoli'', Via Cintia, Monte S. Angelo - 80126 Napoli, Italia. Email: nunzia.gavitone@unina.it} 
}

\begin{document}

%\doublespacing
\maketitle

%\tableofcontents

\begin{abstract} In this paper we  study optimal lower and upper bounds for functionals involving the first Dirichlet eigenvalue $\lambda_{F}(p,\Omega)$ of the anisotropic $p$-Laplacian, $1<p<+\infty$.  Our aim is to enhance how, by means of the $\mathcal P$-function method, it is possible to get several sharp estimates for $\lambda_{F}(p,\Omega)$ in terms of several geometric quantities associated to the domain. The $\mathcal P$-function method is based on a maximum principle for a suitable function involving the eigenfunction and its gradient.

\medskip
\textsc{Keywords:} Dirichlet eigenvalues, anisotropic operators, optimal estimates

\medskip
\textsc{MSC 2010: 35P30, 49Q10} 
\end{abstract}

\section{Introduction}

Given a bounded domain $\Omega\subset \R^{N}$ and $p\in]1,+\infty[$, let us consider the first Dirichlet eigenvalue of the anisotropic $p$-Laplacian, that is:
\begin{equation*}
\label{eigint}
\lambda_{F}(p,\Omega)=\min_{\psi\in W_0^{1,p}(\Omega)\setminus\{0\} } \frac{\displaystyle\int_\Omega F(\nabla \psi)^p
     dx }{\displaystyle\int_\Omega|\psi|^p dx},
\end{equation*}
where $F : \mathbb{R}^N \to [0, +\infty[$, $N\ge 2$, is a convex, even,  $1$-homogeneous and $C^{3,\beta}(\mathbb{R}^N\setminus \{0\})$  function such that
$[F^{p}]_{\xi\xi}\text{ is positive definite in } \R^{N}\setminus\{0\}$, $1<p<+\infty$. 
We are interested in the study of optimal lower and upper bounds for functionals involving $\lambda_{F}(p,\Omega)$. In this order of ideas,  our aim is enhance how these estimates may be obtained as a consequence of a maximum principle for a function which involves an eigenfunction and its gradient, namely the so-called $\mathcal P$-function, introduced by L.E. Payne in the case of the classical Euclidean Laplace operator. We refer the reader to the book by Sperb \cite{sp}, and the references therein contained, for a survey on the $\mathcal P$-function method in the Laplacian case and its applications. More precisely, if $u$ is a positive eigenfunction associated to $\lambda_{F}(p,\Omega)$, we introduce the following function:
\begin{equation}
\label{pfi}
\mathcal P :=(p-1)F^p(\nabla u)+\lambda_{F}(p,\Omega)\left(u^p-M^p\right),
\end{equation}
where $M$ is the maximum value of $u$. We show that the function $\mathcal P$ verifies a maximum principle in $\overline \Omega$ in order to get a pointwise estimate for the gradient in terms of $u$. This is the starting point in order to prove several useful bounds, involving  quantities which depend on the domain $\Omega$. As a matter of fact, the use of the $\mathcal P$-function method in the anisotropic setting has been studied in the recent paper \cite{dpgg}. Here the authors consider the $p$-anisotropic torsional rigidity 
\begin{equation}
\label{t}
T^{p-1}_{F}(p,\Omega)=\sup_{\psi\in W_0^{1,p}(\Omega)\setminus\{0\}} \frac{\displaystyle\left(\int_\Omega|\psi| dx\right)^{p}}{\displaystyle\int_\Omega F(\nabla \psi)^p
     dx },
\end{equation}
and show  optimal bounds for two functionals involving $T_{F}(p,\Omega)$ and some geometric quantities related to the domain. In this spirit, we aim to analyse the case of the eigenvalue problem. Given a convex, bounded domain $\Omega\subset \R^N$ our main results can be summarized as follows. We prove the anisotropic version of Hersch inequality for $\lambda_{F}(p,\Omega)$, namely that
\begin{equation}
\label{Hi}
\lambda_F(p,\Omega) \ge \left(\frac{\pi_{p}}{2}\right)^{p}\frac{1}{R_F(\Omega)^{p}},
\end{equation} 
where $R_F(\Omega)$ is the anisotropic inradius defined in Section 2 and
\begin{equation}
\label{defpp}
\pi_p:=2\int_0^{(p-1)^{1/p}} \frac{dt}{[1-t^p/(p-1)]^{1/p}}
=2\pi\frac{(p-1)^{1/p}}{p\sin\frac\pi p}.
\end{equation}
As regards the Euclidean setting, for $p=2$ the inequality \eqref{Hi} has been proved by Hersch \cite{he} and improved  in \cite{pr}, and generalized for any $p$ in \cite{kaj} (see also \cite{po}). In the general anisotropic case or $p=2$ it has been studied in \cite{bgm,wxpac}.
Another consequence of the maximum principle for $\mathcal P$ that we obtain is following inequality: 
\begin{equation}
\label{payneineqi}
 \left(\frac{p-1}{p}\right)^{p-1}\left(\frac{\pi_{p}}{2}\right)^{p} \le \lambda_{F}(p,\Omega) M_{v_{\Omega}}^{p-1},
\end{equation}
where $v_{\Omega}$ is the positive maximizer of \eqref{t} such that 
\begin{equation*}
  \label{tors0i}
T_F(p,\Omega) =\int_\Omega v_\Omega \, 
    dx,
\end{equation*}
and $M_{v_{\Omega}}$ is the maximum of $v_{\Omega}$.
Inequality \eqref{payneineqi}, in the Euclidean case ($p=2$), has been first proved in \cite{pay}, and then studied also, for instance, in \cite{sp, vdb, hlp}.%, where also the optimality of the constant is proved.

Last main result we show is the following. Let $u$ be a first eigenfunction relative to $\lambda_{F}(p,\Omega)$ and consider the so-called anisotropic ``{efficiency ratio}''
\begin{equation*}
\label{effi}
E_{F}(p,\Omega):=\frac{\|u\|_{p-1}}{|\Omega|^{\frac{1}{p-1}}\|u\|_{\infty}}. %\le \frac{1}{(p-1)^{1-\frac{1}{p}}} \left(\frac{2}{\pi_{p}}\right)^{\frac{1}{p-1}}.
\end{equation*}
Then we prove that
\begin{equation}
\label{epi}
E_{F}(p,\Omega)\le \frac{1}{(p-1)^{\frac{1}{p}}} \left(\frac{2}{\pi_{p}}\right)^{\frac{1}{p-1}},
\end{equation}
where $\pi_p$ is defined in \eqref{defpp}. In the Euclidean case and $p=2$, this inequality is due to Payne and Stakgold, who proved it in \cite{ps}. 

Finally, we  show the optimality in \eqref{Hi} and \eqref{payneineqi}, while the optimality of \eqref{epi} in the class of convex sets is still an open problem. 

As matter of fact, the convexity assumption in \eqref{Hi}, \eqref{payneineqi} and \eqref{epi} can be weakened, being them valid also in the case of smooth domains with anisotropic nonnegative mean curvature (see Section 2 for the definition).

In the present paper we also emphasize the relation of $\lambda_{F}(p,\Omega)$ with the so-called anisotropic Cheeger constant $h_F(\Omega)$ (See Section 3 for the definition). Indeed, in the class of convex sets we prove the validity of a Cheeger type inequality for $\lambda_{F}$, as well as a reverse Cheeger inequality.

The paper is organized as follows. 

In Section 2 we fix the notation and recall some basic facts regarding the eigenvalue problem for the anisotropic $p$-Laplacian, and the torsional rigidity $T_{F}(p,\Omega)$.

Section 3 is devoted to the study of $h_F(\Omega)$. More precisely we recall the definition, the main properties and we prove optimal lower and upper bounds for $h_F(\Omega)$ in terms of the anisotropic inradius $R_F(\Omega)$ of a convex set $\Omega$. In Section 4 we prove that the $\mathcal P$-function in \eqref{pfi} verifies a maximum principle and finally in the last section we prove the quoted results \eqref{Hi}, \eqref{payneineqi}, \eqref{epi} and a reverse Cheeger inequality investigating also the optimality issue.

\section{Notation and preliminaries}%\label{notprel}
Throughout the paper we will consider a convex even 1-homogeneous function 
\[
\xi\in  \R^{N}\mapsto F(\xi)\in [0,+\infty[,
\] 
that is a convex function such that
\begin{equation}
\label{eq:omo}
F(t\xi)=|t|F(\xi), \quad t\in \R,\,\xi \in  \R^{N}, 
\end{equation}
 and such that
\begin{equation}
\label{eq:lin}
a|\xi| \le F(\xi),\quad \xi \in  \R^{N},
\end{equation}
for some constant $a>0$. The hypotheses on $F$ imply there exists $b\ge a$ such that
\begin{equation*}
%\label{upb}
F(\xi)\le b |\xi|,\quad \xi \in  \R^{N}.
\end{equation*}
Moreover, throughout the paper we will assume that $F\in C^{3,\beta}(\mathbb R^{N}\setminus \{0\})$, and
\begin{equation}
\label{strong}
[F^{p}]_{\xi\xi}(\xi)\text{ is positive definite in } \R^{N}\setminus\{0\},
\end{equation}
with $1<p<+\infty$. 

The hypothesis \eqref{strong} on $F$ ensures that the operator 
\[
\Qp u:= \dive \left(\frac{1}{p}\nabla_{\xi}[F^{p}](\nabla u)\right)
\] 
is elliptic, hence there exists a positive constant $\gamma$ such that
\begin{equation*}
\frac1p\sum_{i,j=1}^{n}{\nabla^{2}_{\xi_{i}\xi_{j}}[F^{p}](\eta)
  \xi_i\xi_j}\ge
\gamma |\eta|^{p-2} |\xi|^2, 
\end{equation*}
for any $\eta \in \Rn\setminus\{0\}$ and for any $\xi\in \Rn$. 
%\begin{rem}
%We stress that for $p\ge 2$ the condition  
%\begin{equation*}
%\nabla^{2}_{\xi}[F^{2}](\xi)\text{ is positive definite in } \R^{N}\setminus\{0\},
%\end{equation*}
%implies \eqref{strong}.
%\end{rem}
The polar function $F^o\colon \R^N \rightarrow [0,+\infty[$ 
of $F$ is defined as
\begin{equation*}
F^o(v)=\sup_{\xi \ne 0} \frac{\langle \xi, v\rangle}{F(\xi)}. 
\end{equation*}
 It is easy to verify that also $F^o$ is a convex function
which satisfies properties \eqref{eq:omo} and
\eqref{eq:lin}. Furthermore, 
\begin{equation*}
F(v)=\sup_{\xi \ne 0} \frac{\langle \xi, v\rangle}{F^o(\xi)}.
\end{equation*}
From the above property it holds that
\begin{equation}
\label{imp}
|\langle \xi, \eta\rangle| \le F(\xi) F^{o}(\eta), \qquad \forall \xi, \eta \in  \R^{N}.
\end{equation}
The set
\[
\mathcal W = \{  \xi \in  \R^N \colon F^o(\xi)< 1 \}
\]
is the so-called Wulff shape centered at the origin. We put
$\kappa_N=|\mathcal W|$, where $|\mathcal W|$ denotes the Lebesgue measure
of $\mathcal W$. More generally, we denote with $\mathcal W_r(x_0)$
the set $r\mathcal W+x_0$, that is the Wulff shape centered at $x_0$
with measure $\kappa_Nr^N$, and $\mathcal W_r(0)=\mathcal W_r$.

%We observe that $F$ is the support function  of $\overline{\mathcal W}$. 

The following properties of $F$ and $F^o$ hold true:
\begin{gather*}
\label{prima}
 \langle F_\xi(\xi) , \xi \rangle= F(\xi), \quad  \langle F_\xi^{o} (\xi), \xi \rangle
= F^{o}(\xi),\qquad \forall \xi \in
 \R^N\setminus \{0\}
 \\
 \label{seconda} F(  F_\xi^o(\xi))=F^o(  F_\xi(\xi))=1,\quad \forall \xi \in
 \R^N\setminus \{0\}, 
\\
\label{terza} 
F^o(\xi)   F_\xi( F_\xi^o(\xi) ) = F(\xi) 
 F_\xi^o\left(  F_\xi(\xi) \right) = \xi\qquad \forall \xi \in
 \R^N\setminus \{0\}. 
\end{gather*}

\subsection{Anisotropic mean curvature}
Let $\Omega$ be a $C^{2}$ bounded domain, and $n_{\mathcal E}(x)$ be the unit outer normal at $x\in\de\Omega$, and let $u\in C^{2}(\overline\Omega)$ such that $\Omega_t=\{u>t\}$, $\de \Omega_t=\{u=t\}$ and $\nabla u\ne 0$ on $\de \Omega_t$.  The anisotropic outer normal $n_{F}$ to $\de \Omega_t$ is given by
  \[
  n_F(x)= F_{\xi}(n_{\mathcal E}(x))= F_{\xi}\left(-\nabla u\right),\quad x\in \de \Omega.
  \]
It holds 
 \[
  F^o(n_F)=1.  
  \]
  The anisotropic mean curvature of $\partial \Omega_t$ is  defined as
  \begin{equation*}
  \label{H_F} 
  \mathcal H_{F}(x)= \dive\left( n_{F}(x)\right)=
  \dive\left[ \nabla_{\xi}
    F\left(-{\nabla u(x)}\right) \right], \quad x\in
  \de \Omega_t.
  \end{equation*}
It holds that
\begin{equation}
\label{der_nf}
\frac{\de u}{\de n_F}=\nabla u\cdot n_F= \nabla u\cdot F_{\xi}(-\nabla u)= -F(\nabla u). 
\end{equation}
In \cite{dpgg} it has been proved that for a smooth function $u$, on its level sets $\{u=t\}$ it holds
\begin{equation}
\label{xiaformula}
\mathcal Q_p u=F^{p-2}(\nabla u) \left(  \frac{\de u}{\de n_F}\mathcal H_F+(p-1)\frac{\de^2 u}{\de n_F^2}\right).
\end{equation}
%where $\frac{\de u}{\de n_F}=\nabla u\cdot n_F$ and $\mathcal H_F$ is the anisotropic mean curvature of $\de\Omega_{t}$ as defined in \eqref{H_F}.

Finally we recall the definition of the anisotropic distance from the boundary and the anisotropic inradius.

Let us consider a bounded domain $\Omega$, that is a connected open set of $ \R^N$, with non-empty boundary. 

The anisotropic distance of $x\in\overline\Omega$ to the boundary of $\de \Omega$ is the function 
\begin{equation*}
%\label{defdist}
d_{F}(x)= \inf_{y\in \de \Omega} F^o(x-y), \quad x\in \overline\Omega.
\end{equation*}

We stress that when $F=|\cdot|$ then $d_F=d_{\mathcal{E}}$, the Euclidean distance function from the boundary.
It is not difficult to prove that $d_{F}\in W_{0}^{1,\infty}(\Omega)$ and, using the property of $F$ we have
\begin{equation}
  \label{Fd}
  F(\nabla d_F(x))=1 \quad\text{a.e. in }\Omega.
\end{equation}
Moreover we recall that $\Omega$ is convex the anisotropic distance function is concave.

The quantity%
%Furthermore, if $\de\Omega$ is $C^2$, then $d_F$ is $C^2$ in a suitable neighborhood of $\de \Omega$ in $\overline\Omega$ %\textcolor{red}{Controllare Crasta Malusa per la regolarit\`a (2.2, pag 5729). Richiedono $\{F<1\}$ sia $C^{2}$ a curvatura strettamente positiva. A noi sembra che questo non serva, si lavora solo sulle derivate prime per il conto di $\Delta_{F}d_{F}$}
%(see \cite{cm07}).
\begin{equation}
\label{inrad}
R_{F}(\Omega)=\sup \{d_{F}(x),\; x\in\Omega\},
\end{equation}
is called the anisotropic inradius of $\Omega$.

For further properties of the anisotropic distance function we refer
the reader to \cite{cm07}.

\subsection{The first Dirichlet eigenvalue  for $\mathcal Q_p$}
Let $\Omega$ be a bounded open set in $\R^{N}$, $N\ge 2$, $1<p<+\infty$, and consider the eigenvalue problem
\begin{equation}
\label{eigpb}
\left\{
\begin{array}{ll}
-\Qp u=\lambda |u|^{p-2}u & \text{in } \Omega \\
u=0 &\text{on } \partial\Omega.
\end{array}
\right.
\end{equation}

The smallest eigenvalue, denoted  by $\lambda_{F}(p,\Omega)$, has the following well-known variational characterization:
\begin{equation}
\label{rayleigh}
\lambda_{F}(p,\Omega)=\min_{\varphi\in W^{1,p}_{0}(\Omega)\setminus\{0\}} \frac{\ds\int_\Omega F^p(\nabla \varphi)\ dx}{\ds\int_\Omega |\varphi|^p\ dx}.
\end{equation}
The following two results which enclose the main properties of $\lambda_{F}(p,\Omega)$ hold true. We refer the reader, for example, to \cite{bfk,dgp2}.
\begin{thm}
%\label{th_prop1}
 If $\Omega$ is a bounded open set in $\R^{N}$, $N\ge 2$, there exists a function $u_{1}\in C^{1,\alpha}(\Omega)\cap C(\overline{\Omega})$ which achieves the minimum in \eqref{rayleigh}, and satisfies the problem \eqref{eigpb} with $\lambda=\lambda_{F}(p,\Omega)$. Moreover, if $\Omega$ is connected, then $\lambda_{F}(p,\Omega)$ is simple, that is the corresponding eigenfunctions are unique up to a multiplicative constant, and the first eigenfunctions have constant sign in $\Omega$. 
 \end{thm}
 In the following Proposition the scaling and monotonicity properties of  $\lambda_{F}(p,\Omega)$ are recalled.
\begin{prop}
%\label{fireig}
 Let $\Omega$ be a bounded open set in $\R^{N}$, $N\ge 2$, the following properties hold.
\begin{enumerate}
\item For $t>0$ it holds $\lambda_{F} (p,t\Omega)=t^{-p}\lambda_{F}(p,\Omega)$.
\item If $\Omega_1\subseteq\Omega_2\subseteq\Omega$, then $\lambda_{F} (p,\Omega_1)\ge \lambda_{F} (p,\Omega_2)$.
\item For all $1< p< s< +\infty$ we have $p[\lambda_{F}(p,\Omega)]^{1/p}<s[\lambda_F(s,\Omega)]^{1/s}$.
\end{enumerate}
\end{prop}

\subsection{Anisotropic $p$-torsional rigidity}
In this subsection we summarize some properties of the anisotropic $p$-torsional rigidity. We refer the reader to \cite{dgmana} for further details.

Let $\Omega$ be a bounded domain in $ \R^{N}$, and $1<p<+\infty$. Throughout the paper we will denote by $q$ the H\"older conjugate of $p$,
\[
q:=\frac{p}{p-1}.
\]
Let us consider the torsion problem for the anisotropic $p-$Laplacian
\begin{equation}
\label{pb_tor}
\left\{
\begin{array}{ll}-\mathcal Q_{p}v:=-\dive \left(F^{p-1}(\nabla v) F_\xi (\nabla v)\right)=1  &\text{in}\ \Omega \\
v=0 &\text{on}\ \partial\Omega.
\end{array}
\right.
\end{equation}
 By classical result there exists  a unique solution of \eqref{pb_tor}, that we will always denote by $v_{\Omega}$, which is positive in $\Omega$. Moreover, by \eqref{strong} and being $F\in C^{3}(\R^{n}\setminus \{0\})$, then $v_{\Omega}\in C^{1,\alpha}(\Omega)\cap C^{3}(\{\nabla v_{\Omega}\ne 0\})$ (see \cite{ladyz,tk84}).

The anisotropic $p$-torsional  rigidity of $\Omega$  is
\begin{equation*}
  \label{eq:ptor1}
  T_{F}(p,\Omega)=\int_\Omega F(\nabla v_{\Omega})^p dx = \int_\Omega v_{\Omega} dx.
\end{equation*}

%A characterization of $T_p$ is provided by the equality
%$T_{F}(p,\Omega)=\sigma(\Omega)^{\frac{1}{p-1}}$, where
%$\sigma(\Omega)$ is the best constant in the Sobolev inequality
%\[
%\|\varphi\|_{L^1(\Omega)}^p \le \sigma(\Omega) \|F(\nabla \varphi)\|^p_{L^p(\Omega)},
%\]
%that is
The following variational characterization for $T_{F}(p,\Omega)$ holds
\begin{equation}
  \label{tors0}
T_{F}(p,\Omega)^{p-1} =\max_{\substack{\psi \in
    W_0^{1,p}(\Omega) \setminus\{0\}}}
\dfrac{\left(\displaystyle\int_\Omega |\psi| \, 
    dx\right)^p}{\displaystyle\int_\Omega F(\nabla\psi)^p dx},
\end{equation}
and the solution $v_{\Omega}$ of \eqref{pb_tor} realizes the maximum
in \eqref{tors0}. 

%It is immediate to see that if $\Omega\subset\tilde\Omega$, then
%\begin{equation*}
%\label{monotonia}
%T_{F}(p,\Omega)\le T_p(\tilde\Omega).
%\end{equation*}
By the maximum principle it holds that
\begin{equation*}
\label{monotoniamax}
M_{v_{\Omega}}\le M_{v_{\tilde\Omega}},
\end{equation*}
where $M_{v_{\Omega}}$ is the maximum of the torsion function in $\Omega$.
Finally we recall the following estimates for $M_{v_{\Omega}}$ contained in \cite{dpgg}.
\begin{thm}
Let $\Omega$ be a bounded convex open set in $ \R^{N}$, and $R_F$  the anisotropic inradius defined in  \eqref{inrad}. Then
\begin{equation}
\label{stima_max}
\frac{R_{F}^{q}(\Omega)}{qN^{q-1}}\le M_{v_{\Omega}}\le \frac{R_{F}^{q}(\Omega)}{q}.
\end{equation}
\end{thm}

\section{Anisotropic Cheeger constant}
Let $\Omega$ be an open subset of $\mathbb R^N$. The total variation
of a function $u\in BV(\Omega)$ with respect to $F$ is (see \cite{ab}):
\[
\int_\Omega |\nabla u|_F = \sup\left\{ \int_\Omega u\dive \sigma dx\colon
  \sigma \in C_0^1(\Omega;\R^N),\; F^o(\sigma)\le 1 \right\}.
\]
This yields the following definition of anisotropic perimeter of
$K\subset \R^N$ in $\Omega$:
\begin{equation}
\label{per}
P_F(K) = \int_{\R^{N}}|\nabla\chi_K|_F= \sup\left\{ \int_K
  \dive \sigma dx\colon \sigma \in C_0^1(\R^{N};\R^N),\;
  F^o(\sigma)\le 1 \right\}.
\end{equation}
It holds that
\[
P_F(K)= \int_{\partial^*K} F(n_{\mathcal E}) d\mathcal H^{N-1}
\]
where $\mathcal H^{N-1}$ is the $(N-1)$-dimensional Hausdorff measure
in $\mathbb R^N$, $\partial^*K$ is the reduced boundary of $F$ and
$n_{\mathcal E}$ is the Euclidean unit outer normal to $K$ (see \cite{ab}).

An isoperimetric inequality for the anisotropic perimeter holds,
namely is $\mathcal W_{R}$ is the Wulff shape such that $|\mathcal W_{R}|=|K|$, then
\begin{equation}
  \label{isop}
  P_F(K) \ge P_{F}(\mathcal W_{R})=N\kappa_{N}^{\frac1N}|K|^{1-\frac{1}{N}},
\end{equation}
and the equality holds if and only if $\Omega$ is a Wulff shape
(see for example \cite{bu}, \cite{fomu}, \cite{aflt}). 
The following Lemma will play a key role in order to investigate on optimality issue of the quoted results.
\begin{lem}
\label{fantasticolemma}
Let $\displaystyle\Omega_{a,k}=]-a,a[\times]-k,k[^{N-1}$ a $N$-rectangle in $\R^{N}$, and suppose that $R_{F}(\Omega_{a,k})=aF^{o}(e_{1})$. Then
\begin{equation}
\label{fantastica}	
\lim_{k\to +\infty}\frac{P_{F}(\Omega_{a,k})}{|\Omega_{a,k}|}=\frac{1}{aF^{o}(e_{1})}.
\end{equation}
\end{lem}
\begin{proof}
First observe that (see \cite{dpgg})
\begin{equation}
\label{serena}
F^{o}(e_{1})F(e_{1})=1.
\end{equation}
By definition of anisotropic perimeter we get
\[
\frac{P_{F}(\Omega_{a,k})}{|\Omega_{a,k}|} = \frac{2 (2k)^{N-1}F(e_{1})+O(k^{N-2})}{2^{N}k^{N-1}a},
\]
hence, using \eqref{serena} and passing to the limit we get \eqref{fantastica}.
\end{proof}

The anisotropic Cheeger constant associated to an open bounded set $\Omega\subseteq\R^{N}$ is defined as
\[
h_{F}(\Omega)=\inf_{K\subset \Omega}\frac{P_{F}(K)}{|K|}.
\]
We recall that for a given bounded open set in $\R^{N}$, the Cheeger inequality states that
\begin{equation}
\label{cheeger}
\lambda_{F}(p,\Omega)\ge \left(\frac{h_{F}(\Omega)}{p}\right)^{p}.
\end{equation}
This inequality, well-known in the Euclidean case after the paper by Cheeger (\cite{ch}) in the case $p=2$, has been proved in \cite{kn08} in the anisotropic case. We refer the reader to \cite{pa2} and the reference therein contained for a survey on the properties of the Cheeger constant in the Euclidean case.

It is known (see \cite{kn08} and the references therein) that if $\Omega$ is a Lipschitz bounded domain, there exists a Cheeger set, that is a set $K_{\Omega}$ for which
\[
h_{F}(\Omega)=\frac{P_{F}(K_{\Omega})}{|K_{\Omega}|}.
\]
When $\Omega=\mathcal W_R$ we immediately get that $K_{\mathcal W_{R}}=\mathcal W_{R}$ and
\begin{equation}
\label{cw}
h_F(\mathcal W_R)= \frac{N}{R}.
\end{equation}
We observe that  usually the Cheeger set $K_{\Omega}$  is not unique, nevertheless  $\Omega$ is convex (see for instance \cite{ac,ccmn,kn08}).
\begin{thm}
If $\Omega$ is a bounded convex domain, there exists a unique convex Cheeger set. 
\end{thm}
The  next results give an upper  bound for the Cheeger constant in terms of the anisotropic inradius of $\Omega$.
\begin{prop}
If $\Omega$ is a bounded open set in $\R^{N}$, then
%\begin{itemize}
\begin{equation}
\label{hru}
 h_{F}(\Omega)\le  \frac{N}{R_{F}(\Omega)}.
\end{equation}
Moreover the equality  holds if $\Omega$ is a Wulff shape.
\end{prop}
\begin{proof}
By definition the constant $h_{F}(\Omega)$ is monotone decreasing respect the set inclusion. Then by \eqref{cw} and the definition of anisotropic inradius, we get the inequality  \eqref{hru}.
\end{proof}
As regards a lower bound for the anisotropic Cheeger constant in terms of the inradius of $\Omega$, we have the following.
\begin{prop}
If $\Omega$ is a bounded open convex set in $\R^{N}$, then
%\begin{itemize}
\begin{equation}
\label{hrl}
\frac{1}{R_{F}(\Omega)}\le h_{F}(\Omega).
\end{equation}
Moreover, the inequality is optimal for a suitable sequence of  $N$-rectangular domains.
\end{prop}
\begin{proof}
Using \eqref{Fd}, \eqref{eq:omo} and the coarea formula, we have for a bounded convex set $K\subseteq \Omega$ that 
\begin{multline*}
|K|= \int_{K} F(\nabla d_F) dx = 
\int_{0}^{R_{F}(K)
} dt \int_{d_{F}=t} F(n_{\mathcal E})d\sigma=\\=
\int_0^{R_F(K)} P_F(\{d_F\le t\}) dt \le P_F(K) R_F(K).
\end{multline*}
Hence, being $K$ convex $R_{F}(K)\le R_{F}(\Omega)$, then
\[
\frac{P_{F}(K)}{|K|} \ge \frac{1}{R_{F}(\Omega)}.
\]
Passing to the infimum on $K$, we get \eqref{hrl}. As regards the optimality, it follows immediately from \eqref{fantastica}.
\end{proof}
More generally for convex sets it holds the following result (see also for instance \cite{dgmaan}, where the case $N=2$ is given with a different proof).
\begin{prop}
If $\Omega \subset \R^{N}$ is a bounded, open, convex set, then
\begin{equation*}
%\label{pmru}
 \frac{P_{F}(\Omega)}{|\Omega|}\le  \frac{N}{R_{F}(\Omega)}.
\end{equation*}
For the Wulff shape the equality holds.
\end{prop}
\begin{proof}
Let   $x_{0} \in \Omega$ such that  $R_{F}(\Omega)=d_{F}(x_{0})$. By the concavity of $d_{F}$, we have 
\begin{equation*}
\label{hk}
d_{F}(x_{0}) - d_{F}(x)\le -\nabla d_{F}(x) \cdot (x-x_{0}) = n_{\mathcal{E}}(x) \cdot (x-x_{0})|\nabla d_{F}(x)|.
\end{equation*} 
Hence for $x\in\de\Omega$, it holds that $R_{F}(\Omega) \le n_{\mathcal{E}}(x) \cdot (x-x_{0})|\nabla d_{F}(x)|$. By the divergence theorem, and observing also that $F(n_{\mathcal E})=\frac{1}{|\nabla d_{F}|}$ we have
\begin{multline*}
|\Omega|= \frac{1}{N} \int_{\Omega} \dive (x-x_{0})\, dx=\frac{1}{N}\int_{\de \Omega} (x-x_{0}) \cdot n_{\mathcal E}(x)d\sigma\ge \\ \ge \frac{R_{F}(\Omega)}{N}\int_{\de\Omega}\frac{1}{|\nabla d_{F}(x)|}d\sigma = \frac{R_{F}(\Omega)P_{F}(\Omega)}{N},
\end{multline*}
and this completes the proof.
\end{proof}
\begin{rem}
We observe that the equality in the inequality of the previous proposition holds, in general, also for other kind of convex sets. For example, if $N=2$ and $F=\mathcal E$ the equality holds for circles with two symmetrical caps (see for instance \cite{santalo}).
\end{rem}

An immediate consequence of the anisotropic isoperimetric inequality is the following
\begin{thm}
Let $\Omega$ be a bounded open set. Then
\begin{equation}
\label{fkh}
h_{F}(\Omega)\ge h_{F}(\mathcal W_R),
\end{equation}
where $\mathcal W_R$ is the Wulff shape such that $|\Omega|=|\mathcal W_R|$, and the equality holds if and only if $\Omega$ is a Wulff shape.
\end{thm}
\begin{proof}
Let $K\subseteq\Omega$. Then, if $|\mathcal W_{r}|=|K|$, by \eqref{isop} we have:
\[
\frac{P_{F}(K)}{|K|} \ge \frac{P_{F}(\mathcal W_{r})}{|\mathcal W_{r}|}\ge \frac{P_{F}(\mathcal W_R)}{|\mathcal W_{R}|}=h(\mathcal W_R).
\]
Passing to the infimum on $K$, we get the result.
\end{proof}
\begin{rem}
Let $\Omega$ be an open, bounded set of $\R^{N}$ inequality \eqref{hru} implies 
\begin{equation}
\label{stab}
h_F(\Omega)-h_F(W_R) \le N\left(\frac{1}{R_{F}(\Omega)}-\frac{1}{R} \right).
\end{equation}
When $\Omega$ is convex \eqref{stab} can be read as  a stability result for \eqref{fkh}.
\end{rem}

In \cite{dgmana} it is proved the following upper bound for  $\lambda_{F}(p,\Omega)$ in terms of volume and anisotropic perimeter of $\Omega$ for convex domains.
\begin{thm}
Let $\Omega \subset \R^N$ be a bounded, convex,  open set. Then
\begin{equation}
\label{pol}
\lambda_{F}(p,\Omega) \le \left( \frac{\pi_p}{2}\right)^p \displaystyle \left(\frac{P_F(\Omega)}{|\Omega|}\right)^p,
\end{equation}
where $\pi_p$ is defined in \eqref{defpp} and $P_F(\Omega)$ is the anisotropic perimeter of $\Omega$ defined in \eqref{per}.
\end{thm}

The following reverse anisotropic Cheeger inequality holds (see \cite{pa} for the Euclidean case with $N=p=2$).
\begin{prop}
Let $\Omega \subset \R^N$, $N \ge 2$, be a bounded open convex open set. Then
\begin{equation}
\label{rc}
\lambda_{F}(p,\Omega)\le \left( \frac{\pi_p}{2}\right)^p h_F(\Omega)^p,
\end{equation}
where $\pi_p$ is defined in \eqref{defpp}. 
\end{prop}
\begin{proof}
Let $K_{\Omega}\subseteq\Omega$ be the convex Cheeger set of $\Omega$. Being the $\lambda_{p}(\cdot)$ monotone decreasing by set inclusion, then by \eqref{pol} we have
\[
\lambda_{F}(p,\Omega) \le \lambda_{F}(p,K_{\Omega}) \le \left(\frac{\pi_{p}}{2}\right)^{p} \left(\frac{P_{F}(K_{\Omega})}{|K_\Omega|}\right)^{p}=\left( \frac{\pi_p}{2}\right)^p h_F(\Omega)^p.
\]
\end{proof}
The equality sign holds in the limiting case when $\Omega$ approaches a slab. This will be shown in Theorem \ref{optimality}.

\section{The $\mathcal{P}$-function}
In order to give some sharp lower bound for  $\lambda_{F}(p,\Omega)$ we will use the so-called $\mathcal P$-function method. Let us consider the general problem 
\begin{equation}
\label{pb_gen}
\begin{cases}
-\mathcal Q_{p}w =f(w) &\text{in }\Omega\\
w=0 &\text{on }\de\Omega,
\end{cases}
\end{equation}
where $f$ is a nonnegative $C^{1}(0,+\infty)\cap C^{0}([0,+\infty[)$ function, and define
\begin{equation*}
\label{pfunct}
\mathcal P(x):=\frac{p-1}{p}F^p(\nabla w(x))-\int_{w(x)}^{\max_{\bar \Omega} w}f(s)ds 
\end{equation*}
 The following result is proved in \cite{cfv}.
\begin{prop}
\label{cfvprop}
Let $\Omega$ be a domain in $ \R^{N}$, $ N\ge 2$, and $w\in W_{0}^{1,p}(\Omega)$ be a solution of \eqref{pb_gen}. Set
\[
d_{ij}:= \frac{1}{F(\nabla u_{\Omega})}\de_{\xi_{i}\xi_{j}}\left[\frac{F^{p}}{p}\right](\nabla u_{\Omega}),
\]
Then it holds that
\[
\left( d_{ij}\mathcal P_{i} \right)_{j} -b_{k}\mathcal P_{k} \ge 0 \quad\text{ in }\{\nabla u_{\Omega}\ne 0 \}
\]
where
\begin{equation*}
b_{k}=\frac{p-2}{F^{3}(\nabla w)} F_{\xi_{\ell}}(\nabla w)\mathcal P_{x_{\ell}} F_{\xi_{k}}(\nabla w)+\frac{2p-3}{F^{2}(\nabla w)} \left(   \frac{F_{\xi_{k}\xi_{\ell}}(\nabla w)\mathcal P_{x_{\ell}}}{p-1} - f(w) F_{\xi_{k}}(\nabla w)\right)
\end{equation*}
\end{prop}
As a consequence of the previous result we get the following maximum principle for $\mathcal P$.
\begin{thm}
\label{prsperb}
Let $\Omega$ be a bounded $C^{2}$ domain in $ \R^{N}$, $ N \ge 2$, with nonnegative anisotropic mean curvature $\mathcal H_{F}\ge 0$ on $\de\Omega$, and $w>0$ be a solution to the problem  \eqref{pb_gen}, 
%\[
%\begin{cases}
%-\mathcal Q_{p}w =f(w) &\text{in }\Omega\\
%w=0 &\text{on }\de\Omega.
%\end{cases}
%\] 
then
\begin{equation}
\label{sperbineq}
\mathcal P(x)=\frac{p-1}{p}F^p(\nabla w(x))-\int_{w(x)}^{\max_{\bar \Omega} w}f(s)ds \le 0  \quad \text{ in } \overline{\Omega},
\end{equation}
that is  the function $\mathcal P$ achieves its maximum at the points $x_M \in \Omega$ such that  $w(x_M)=\max_{\bar\Omega} w$.
\end{thm}
\begin{proof}
Let us denote by $\mathcal C$ the set of the critical points of $w$, that is $\mathcal C=\{x \in \overline\Omega \colon \nabla w(x) =0\}$. Being $\de \Omega$ $C^{2}$, by the Hopf Lemma (see for example \cite{ct}), $\mathcal C\cap \de \Omega=\emptyset$.

Applying Proposition \ref{cfvprop}, the function  $\mathcal P$ verifies a maximum principle in the open set $\Omega \setminus \mathcal C$. Then we have
\[
\max_{\overline{\Omega \setminus \mathcal C}} \mathcal P = \max_{\de\left(\Omega \setminus \mathcal C\right)}\mathcal P.
\]
Hence one of the following three cases occur:
\begin{enumerate}
\item the maximum point of $\mathcal P$ is on $\de \Omega$;
\item the maximum point of $\mathcal P$ is on $\mathcal C$;
\item the function $\mathcal P $ is constant in $\overline \Omega$.
\end{enumerate}
In order to prove the theorem we have to show that  statement 1 cannot happen.
Let us compute the derivative of $\mathcal P$ in the direction of the anisotropic normal $n_F$,  in the sense of  \eqref{der_nf}. Hence on $\de\Omega$ we get
\begin{multline*}
\label{derivata}
\frac{\de \mathcal P}{\de n_F}= \dfrac{p-1}{p}\frac{\de}{\de n_F}\left(-\frac{\de w}{\de n_F}\right)^p+f(w)\frac{\de w}{\de n_F} =-(p-1) \left(-\frac{\de w}{\de n_F}\right)^{p-1} \frac{\de^2 w}{\de n_F^2}+f(w)\frac{\de w}{\de n_F}=\\= -F(\nabla w) \mathcal Q_p[w]-F^{p-1}(\nabla w) \mathcal H_F-f(w)F(\nabla w)=-F^{p-1}(\nabla w) \mathcal H_F,
\end{multline*}
where last identity follows by \eqref{xiaformula}. On the other hand, if a maximum point $\bar x$ of $\mathcal P$ is on $\de\Omega$, by Hopf Lemma either $\mathcal P$ is constant in $\overline\Omega$, or $\frac{\de \mathcal P}{\de n_F}(\bar x)> 0$.
Hence being $\mathcal H_F\ge 0$ we have a contradiction.
\end{proof}
\begin{rem}
\label{rmp}
Let $\Omega$ be a bounded, open convex set and let  us consider  $u$ a positive eigenfunction  relative to the first eigenvalue $\lambda_{F}(p,\Omega)$ to the problem \eqref{eigpb}. Then denoted by $M= \max_{\overline \Omega}u$,  inequality \eqref{sperbineq} becomes
\begin{equation}
\label{Pl}
(p-1)F^p(\nabla u) \le \lambda_{F}(p,\Omega) \left( M^p-u^p\right) \quad \text{ in } \overline{\Omega}
\end{equation}
Integrating over $\Omega$ in both sides of \eqref{Pl} and recalling that $u$ satisfies problem \eqref{eigpb}, we get 
 \begin{equation*}
 \label{mp}
 \int_{\Omega}u^p \le \displaystyle \frac{M^p|\Omega|}{p}.
 \end{equation*}
\end{rem}
By the definition of $\pi_p$ in \eqref{defpp}, we have 
\[
\frac{\pi_{p}}{2}=\int_{0}^{(p-1)^{\frac1p}} \left[ 1-\frac{t^{p}}{p-1} \right]^{-\frac 1 p}\,dt=\int_{0}^{M(p-1)^{\frac1p}}\left(M^{p}-\frac{t^{p}}{p-1}\right)^{-\frac1p}\,dt
\]
where $u$ is a first positive eigenfunction of $-\mathcal Q_{p}$. % and, in order to simplify the notation, $M=\max_{\overline\Omega}u$. 
 Let us consider the following function
\begin{equation*}
\label{phi}
\Phi(s)=\left(\frac{\pi_{p}}{2}\right)^{\frac{p}{p-1}}-
\left(\int_{s(p-1)^{\frac 1p}}^{M(p-1)^{\frac1p}}\frac{dt}{\left(M^{p}-\frac{t^{p}}{p-1}\right)^{\frac1p}}\right)^{\frac{p}{p-1}}, \quad s \in [0,M].
\end{equation*}
\begin{prop}
\label{efrat}
Let $\Omega$ be a bounded $C^{2}$ domain in $ \R^{N}$, $ N \ge 2$, with nonnegative anisotropic mean curvature $\mathcal H_{F}\ge 0$ on $\de\Omega$, then the following inequalities hold
\begin{equation}
\label{disphi}
\Phi(u(x))\le \frac{p}{p-1}\lambda_{F}(p,\Omega)^{\frac{1}{p-1}}v_{\Omega}(x),
\end{equation}
and
\begin{equation}
\label{dernorm}
\Phi'(u)F(\nabla u) \le \frac{p}{p-1}\lambda_{F}(p,\Omega)^{\frac{1}{p-1}}F(\nabla v_{\Omega})\quad\text{on }\de\Omega,
\end{equation}
where $v_{\Omega}$ is the stress function of $\Omega$. 
\end{prop}
\begin{proof}
In order to prove \eqref{disphi}, we will show that 
\begin{equation}
\label{confronto}
-\mathcal Q_{p}[\Phi] \le -\mathcal Q_{p}\left[\frac{p}{p-1}\lambda_{F}(p,\Omega)^{\frac{1}{p-1}}v_{\Omega}\right]=\lambda_{F}(p,\Omega)\left(\frac{p}{p-1}\right)^{p-1}.
\end{equation}
By the comparison principle, being $\Phi(u)=v_{\Omega}=0$ on $\de \Omega$, then \eqref{disphi} holds.

Denoting by $\varphi(u)=\displaystyle\int_{u(p-1)^{\frac1p}}^{{M(p-1)^{\frac1p}}}
{\left(M^{p}-\frac{t^{p}}{p-1}\right)^{-\frac1p}}dt
$, we have:
\[
\Phi'(u) = q(p-1)^{\frac1p}\varphi(u)^{q-1}\left(M^{p}-u^{p}\right)^{-\frac{1}{p}},
\]
and
\begin{align*}
&\Phi''(u)=\\&=-q(q-1)(p-1)^{\frac2p}\varphi(u)^{q-2}\left(M^{p}-u^{p}\right)^{-\frac2p}+q(p-1)^{\frac1p}\varphi(u)^{q-1}\left(M^{p}-u^{p}\right)^{-\frac1p-1}u^{p-1}=\\&=
q(p-1)^{\frac1p}\varphi(u)^{q-1}\left(M^{p}-u^{p}\right)^{-\frac1p}
\left[\frac{u^{p-1}}{M^{p}-u^{p}}- (q-1)(p-1)^{\frac1p}\varphi(u)^{-1}\frac{1}{\left(M^{p}-u^{p}\right)^{\frac1p}}\right]\\
&=\Phi'(u)
\left[\frac{u^{p-1}}{M^{p}-u^{p}}- (q-1)(p-1)^{\frac1p}\varphi(u)^{-1}\frac{1}{\left(M^{p}-u^{p}\right)^{\frac1p}}\right]=\Phi'(u)\Psi(u),
\end{align*}
where we denoted last square bracket with $\Psi(u)$.
\begin{multline*}
Q_{p}\Phi(u)=\dive\left[(\Phi')^{p-1}F(\nabla u)^{p-1}F_{\xi}(\nabla u)\right]=\\=(\Phi')^{p-1}Q_{p}u+(p-1)(\Phi')^{p-2}\Phi''(u) F(\nabla u)^{p}
=\\=
(\Phi')^{p-1}\left[-\lambda_{F}(p,\Omega) u^{p-1}+(p-1)F(\nabla u)^{p}\Psi(u)\right].
\end{multline*}
To prove the claim we need to show that \eqref{confronto} holds, that is
\[
(\Phi')^{p-1}\left[-\lambda_{F}(p,\Omega) u^{p-1}+(p-1)F(\nabla u)^{p}\Psi(u)\right]+q^{p-1}\lambda_{F}(p,\Omega) \ge 0.
\]
Substituting, we get:
\begin{multline*}
-\lambda_{F}(p,\Omega) u^{p-1}+(p-1)F(\nabla u)^{p} \left[\frac{u^{p-1}}{M^{p}-u^{p}}- (q-1)(p-1)^{\frac1p}\varphi(u)^{-1}\frac{1}{\left(M^{p}-u^{p}\right)^{\frac1p}}\right]+\\+\frac{\lambda_{F}(p,\Omega)[M^{p}-u^{p}]^{\frac{p-1}{p}}}{(p-1)^{\frac1q}\varphi(u)}=\\
=\left\{
(p-1)^{-\frac{1}{q}}\varphi(u)^{-1}\left[M^{p}-u^{p}\right]^{1-\frac1p}
-u^{p-1}\right\}
\left[
\lambda_{F}(p,\Omega)-\frac{(p-1)F(\nabla u)^{p}}{M^{p}-u^{p}}
\right]
\end{multline*}
The function in the last square brackets is nonnegative, by \eqref{sperbineq}. To conclude, we show that the function
\[
B(u):=\left[M^{p}-u^{p}\right]^{1-\frac1p}
-(p-1)^{1-\frac{1}{p}}u^{p-1}\varphi(u)
\]
is nonnegative, and this is true, being $B(M)=0$ and $B'\le 0$. This concludes the proof of \eqref{disphi}. Finally by computing the derivative of $\Phi$ with respect to the anisotropic normal $ n_{F}$ on $\de\Omega=\{u=0\}$, we have:
\[
\frac{\de \Phi}{\de n_{F}} = \nabla \Phi \cdot F_{\xi}(-\nabla u)=-\Phi'(u)F(\nabla u)\quad\text{on }\de\Omega.
\]
Recalling \eqref{disphi}, by Hopf lemma we get 
\[
\frac{\de \Phi}{\de n_{F}} \ge  \frac{p}{p-1}\lambda_{F}(p,\Omega)^{\frac{1}{p-1}} \frac{\de v_{\Omega}}{\de n_{F}} \quad\text{on }\de\Omega,
\]
then
\[
\Phi'(u)F(\nabla u) \le \frac{p}{p-1}\lambda_{F}(p,\Omega)^{\frac{1}{p-1}}F(\nabla v_{\Omega})\quad\text{on }\de\Omega,
\]
which is \eqref{dernorm}, and this concludes the proof of the theorem.
\end{proof}

\section{Applications}
Now we prove several inequalities involving $\lambda_{F}(p,\Omega)$, $R_{F}(\Omega)$, $h_{F}(\Omega)$, $M_{v_{\Omega}}$, $E_{F}(p,\Omega)$. The main estimates that we prove using the $\mathcal P$-function method (Theorem \ref{her}, Proposition \ref{bettercheegerprop}, Theorem \ref{paythm} and Theorem \ref{ue2thm}) are stated for $C^{2}$ bounded domains in $\R^{N}$ with nonnegative anisotropic mean curvature. Actually, for bounded convex sets the $C^{2}$ regularity is not needed. This can be proved approximating $\Omega$ in the Hausdorff distance by an increasing sequence of strictly convex smooth domains contained in $\Omega$. A similar argument has been used, for example, in \cite{dpgg}.

\subsection{Anisotropic Hersch inequality} 
\begin{thm}
\label{her}
Let $\Omega$ be a bounded $C^{2}$ domain in $ \R^{N}$, $ N \ge 2$, with nonnegative anisotropic mean curvature $\mathcal H_{F}\ge 0$ on $\de\Omega$, then the following anisotropic Hersch inequality holds
\begin{equation}
\label{H}
\lambda_{F}(p,\Omega) \ge \left(\frac{\pi_{p}}{2}\right)^{p}\frac{1}{R_F(\Omega)^{p}},
\end{equation}
where $R_F(\Omega)$ is the anisotropic inradius defined in \eqref{inrad}.
\end{thm}
\begin{proof}
Let $u$ be a positive eigenfunction relative to $\lambda_{F}(p,\Omega)$ and $v$ a direction of $\R^N$. Let $M=\max_{\overline{\Omega}}u $. Then by Theorem \ref{prsperb} with $f(w)=\lambda w^{p-1}$ and property \eqref{imp} we have
\begin{equation}
\label{der_dir}
\displaystyle \frac{\de u}{\de v}=\langle \nabla u, v \rangle  \le F(\nabla u) F^o(v) \le   \left(\frac{\lambda_{F}(p,\Omega)}{p-1}\right)^{\frac{1}{p}}\left(M^p-u^p\right)^{\frac{1}{p}}F^o(v).
\end{equation}
Let us denote by $x_M$ the point of $\Omega$ such that $M=u(x_M)$, by $\bar x \in \de \Omega$ the point such that $F^o(x_M- \bar x)=d_F(x_M)$ and by $v$ the direction of the straight line joining the points $x_M$ and $\bar x$. Then by   \eqref{der_dir} and being $F^{o}(\bar x-x_{M}) \le R_{F}(\Omega) $, we get
\begin{multline}
\label{1}
\displaystyle \int_{0}^{M(\Omega)} \displaystyle \frac{1}{\left(M^p(\Omega)-u^p\right)^{\frac{1}{p}}}du \le  \displaystyle  \left( \frac{\lambda_{F}(p,\Omega)}{p-1}\right)^{\frac{1}{p}}F^o(v)|\bar x-x_{M}|=\\= \left( \frac{\lambda_{F}(p,\Omega)}{p-1}\right)^{\frac{1}{p}}F^{o}(\bar x-x_{M})\le  \displaystyle  \left( \frac{\lambda_{F}(p,\Omega)}{p-1}\right)^{\frac{1}{p}}R_{F}(\Omega).
\end{multline}
By definition of \eqref{defpp} by a change of variable we get
\begin{equation}
\label{2}
\int_{0}^{M(\Omega)} \displaystyle \frac{1}{\left(M^p(\Omega)-u^p\right)^{\frac{1}{p}}}du=\displaystyle \frac{1}{(p-1)^{\frac{1}{p}}} \,\,\frac{\pi_p}{2}
\end{equation}
Finally, joining \eqref{1} and \eqref{2}, we get the inequality \eqref{H}.
\end{proof}
The equality sign in \eqref{H} holds in the limiting case when $\Omega$ approaches a slab. This will be shown in Theorem \ref{optimality}.

From the Hersch inequality \eqref{H} and the bound \eqref{hru} it immediately holds the following.
\begin{prop}
\label{bettercheegerprop}
Let $\Omega$ be a bounded $C^{2}$ domain in $ \R^{N}$, $ N \ge 2$, with nonnegative anisotropic mean curvature $\mathcal H_{F}\ge 0$ on $\de\Omega$. Then
\begin{equation}
\label{bettercheeger}
\lambda_{F}(p,\Omega)\ge \left(\frac{\pi_{p}}{2N}\right)^{p} h_{F}^{p}(\Omega).
\end{equation}
\end{prop}
Hence for $p\ge \frac{2N}{\pi_{p}}$ the inequality \eqref{bettercheeger} gives a better constant than \eqref{cheeger}. 

\subsection{An upper bound for the efficiency ratio}
As a consequence of the Theorem \ref{efrat} we obtain the following inequality:
\begin{thm}
\label{paythm}
Let $\Omega$ be a bounded $C^{2}$ domain in $ \R^{N}$, $ N \ge 2$, with nonnegative anisotropic mean curvature $\mathcal H_{F}\ge 0$ on $\de\Omega$, then
\begin{equation}
\label{payneineq}
 \left(\frac{p-1}{p}\right)^{p-1}\left(\frac{\pi_{p}}{2}\right)^{p} \le \lambda_{F}(p,\Omega) M_{v_{\Omega}}^{p-1},
\end{equation}
where $M_{v_{\Omega}}=\max_{\overline{\Omega}}v_{\Omega}$. 
\end{thm}
\begin{proof}
The proof is a direct consequence of Theorem \ref{efrat} and of the definition \eqref{defpp} of $\pi_p$. Indeed by \eqref{disphi} and the explicit expression of $\Phi$, evaluating  both sides  at the maximizer $x_m$ of $u$ we obtain
\begin{equation*}
\label{app}
\left(\frac{p-1}{p}\right)^{p-1} \left(\frac{\pi_{p}}{2}\right)^p \le v_{\Omega}^{p-1}(x_m) \lambda_{F}(p,\Omega) \le \lambda_{F}(p,\Omega) M_{v_{\Omega}}^{p-1}
\end{equation*}
which is the desired inequality \eqref{payneineq}.
\end{proof}
The equality sign in \eqref{payneineq} holds in the limiting case when $\Omega$ approaches a slab. This will be shown in Theorem \ref{optimality}.

\begin{rem}
We observe that  the functional involved in Theorem \ref{paythm} is related to other functionals studied in literature. Indeed it holds that
\begin{equation}
\label{func}
\lambda_F(p,\Omega) \left( \frac{T_F(p,\Omega)}{|\Omega|}\right)^{p-1}\le\lambda_F(p,\Omega)M^{p-1}_{v_{\Omega}} \le\displaystyle \left(\frac{|\Omega|M_{v_{\Omega}}}{T_F(p,\Omega)}\right)^{p-1}.
\end{equation}
The  functional in the left-hand side of \eqref{func} has been studied for example in \cite{vbfnt} for $p=2$ in the Euclidean case. The functional in the right-hand side  of \eqref{func} has been investigated for instance in  \cite{hlp} for $p=2$ in the Euclidean case and  in  \cite{dpgg}  for any $p$ in the anisotropic setting. 
\end{rem}

\begin{rem}
Using  the upper bound in \eqref{stima_max} and \eqref{payneineq} we get directly the anisotropic Hersch inequality for $\lambda_{F}(p,\Omega)$:
\begin{equation*}
\lambda_{F}(p,\Omega) \ge \left(\frac{\pi_{p}}{2}\right)^{p}\frac{1}{R_F(\Omega)^{p}}.
\end{equation*}
%\textcolor{brown}{Moreover repeating  the same argument of Remark 3.4 we get that the equality in \eqref{H2} holds on the strip.}
\end{rem}
Let $u$ be the first eigenfunction relative to $\lambda_{F}(p,\Omega)$ an let us define the anisotropic efficiency ratio  
\begin{equation}
\label{eff}
E_{F}(p,\Omega):=\frac{\|u\|_{p-1}}{|\Omega|^{\frac{1}{p-1}}\|u\|_{\infty}} %\le \frac{1}{(p-1)^{1-\frac{1}{p}}} \left(\frac{2}{\pi_{p}}\right)^{\frac{1}{p-1}}.
\end{equation}
%Now, let us define
%\[
%\psi(u)=-\left(\int_{u(p-1)^{\frac 1p}}^{M(p-1)^{\frac1p}}\frac{dt}{\left(M^{p}-\frac{t^{p}}{p-1}\right)^{\frac1p}}\right)^{\frac{1}{p-1}}.
%\]
We stress that by Remark \ref{rmp} by H\"older inequality, for open bounded convex sets we obtain the following upper bound for \eqref{eff}
\begin{equation}
\label{ue}
E_{F}^p(p,\Omega) \le \frac{1}{p} 
\end{equation}

Actually as a consequence of the Theorem \ref{efrat} we get the following  upper bound for $E_{F}(p,\Omega)$ which in the Euclidean case is due to Payne and Stakgold \cite{ps}.
\begin{thm}
\label{ue2thm}
Let $\Omega$ be a bounded $C^{2}$ domain in $ \R^{N}$, $ N \ge 2$, with nonnegative anisotropic mean curvature $\mathcal H_{F}\ge 0$ on $\de\Omega$, then
\begin{equation}
\label{ue2}
E_{F}(p,\Omega)\le (p-1)^{-\frac{1}{p}} \left(\frac{2}{\pi_{p}}\right)^{\frac{1}{p-1}}.
\end{equation}
\end{thm}
\begin{proof}
Passing to the power $p-1$ in both sides of  \eqref{dernorm}, integrating on $\de \Omega$ and using the equations, by the divergence theorem we have that
\[
(p-1)^{\frac1q} \frac{\pi_{p}}{2} \int_{\Omega}u^{p-1}dx \le M^{p-1} |\Omega|
\]
that gives the following upper bound the ``efficiency ratio'' $E_{p}$:
\[
E_{F}(p,\Omega)=\frac{\|u\|_{p-1}}{|\Omega|^{\frac{1}{p-1}}\|u\|_{\infty}} \le \frac{1}{(p-1)^{\frac{1}{p}}} \left(\frac{2}{\pi_{p}}\right)^{\frac{1}{p-1}}.
\]
\end{proof}

\begin{rem}
We observe that the bound in \eqref{ue2} improves the one given in \eqref{ue}. 
\end{rem}

Finally we are in position  to give the following optimality result.
\begin{thm}
\label{optimality}
The equality sign in  \eqref{hrl}, \eqref{pol}, \eqref{rc},  \eqref{H},  \eqref{payneineq}, and in the upper bound of \eqref{stima_max} holds in the limiting case when $\Omega$ approaches a suitable infinite slab. 
\end{thm}
\begin{proof}
Let $\Omega$ be a bounded open convex set of $\R^{N}$. Then by \eqref{H}, \eqref{rc}, and the definition of $h_{F}$
we get
\begin{equation}
\label{catenone1}
\left(\frac{\pi_{p}}{2}\right)^{p} \le R_{F}(\Omega)^{p}
\lambda_{F}(p,\Omega) \le \left(\frac{\pi_{p}}{2}\right)^{p} \left(h_{F}(\Omega) R_{F}(\Omega)\right)^{p} \le 
\left(\frac{\pi_{p}}{2}\right)^{p} \left( \frac{P_{F}(\Omega)R_{F}(\Omega)}{|\Omega|}\right)^{p},
\end{equation}
and by \eqref{payneineq}, \eqref{pol} and the upper bound in  \eqref{stima_max} 
\begin{equation}
\label{catenone2}
\left(\frac{p-1}{p}\right)^{p-1} \left(\frac{\pi_{p}}{2}\right)^p \le \lambda_{F}(p,\Omega) M_{v_{\Omega}}^{p-1} \le \left(\frac{p-1}{p}\right)^{p-1} \left( \frac{\pi_p}{2}\right)^p \displaystyle \left(\frac{P_F(\Omega) R_{F}(\Omega)}{|\Omega|} \right)^p.
\end{equation}
Choosing $\Omega=\Omega_{a,k}$ in \eqref{catenone1} and \eqref{catenone2} as in Lemma \ref{fantasticolemma}, and passing to the limit we get the required optimality.
\end{proof}
\begin{rem}
For a general planar open convex set in \cite{vdb}, in the Euclidean case, the author proves the following result
\begin{equation*}
\lambda(\Omega)M_{v_{\Omega}} \le \left(\frac{\pi^2}{8}\right) \left(1+7 \cdot 3^{\frac{2}{3}} \left( \frac{W(\Omega)}{d(\Omega)}\right)^\frac{2}{3}\right),
\end{equation*}
where $\lambda(\Omega)$ is the first Dirichlet eigenvalue of $-\Delta$, $d(\Omega)$ denotes the Euclidean diameter and $W(\Omega )$ the width. Then, for planar open convex set and $p=2$, in the Euclidean case the equality in \eqref{payneineq} holds for the sets such that $\frac{W(\Omega)}{d(\Omega)} \to 0$.
\end{rem}

\begin{rem}
The slab is not optimal for $E_{F}(p,\Omega)$. Indeed, if for example $N=p=2$ and $F=\mathcal E=(\sum_{i}x_{i}^{2})^{1/2}$, for any rectangle $R$ it holds that $E_{\mathcal E}(2,R)=\left(\frac{2}{\pi}\right)^{2}$.
\end{rem}

\section*{Acnowledgements}

This work has been partially supported by GNAMPA of INdAM.


\begin{thebibliography}{20}
\bibitem[AC]{ac} F. Alter, V. Caselles, \textit{Uniqueness of the Cheeger set of a convex body}. Nonlinear Anal. 70 (2009), no. 1, 32-44.

\bibitem[AFLT]{aflt}
Alvino A., Ferone V., Lions P.-L., Trombetti G.,
\textit{Convex symmetrization and applications},
Ann. Inst. H. Poincar\'{e} Anal. Non Lin\'{e}aire, 14(2):275--293, 1997.
    
\bibitem[AB]{ab}
M.~Amar and G.~Bellettini.
\newblock \textit{A notion of total variation depending on a metric with discontinuous coefficients}.
\newblock {Ann. Inst. H. Poincar\'{e} Anal. Non Lin\'{e}aire},
  11(1):91--133, 1994.
  
\bibitem[BFK]{bfk} M.~Belloni, V.~Ferone, and B.~Kawohl.
\newblock \textit{Isoperimetric inequalities, {W}ulff shape and related questions for
  strongly nonlinear elliptic operators}.
\newblock {Zeitschrift fur Angewandte Mathematik und Physik (ZAMP)},
  54(5):771--783, 2003.
 
\bibitem[vdB]{vdb} M. van den Berg, \textit{Spectral Bounds for the Torsion Function}, Integr. Equ. Oper. Theory 88 (2017), 387-400

 
\bibitem[vBFNT]{vbfnt} van den Berg M., Ferone V., Nitsch C., Trombetti C., \textit{On P\'olya's Inequality for Torsional Rigidity and First Dirichlet Eigenvalue}, Integr. Equ. Oper. Theory { 86}, 579--600 (2016).
  
\bibitem[B]{bu}
H.~Busemann.
\newblock {The isoperimetric problem for {M}inkowski area}.
\newblock {\em Amer. J. Math.}, 71:743--762, 1949.

\bibitem[BGM]{bgm} Buttazzo G., Guarino Lo Bianco S., Marini M.,  \textit{Sharp estimates for the anisotropic torsional rigidity and the principal frequency}, J. Math. Anal. Appl. to appear, DOI:10.1016/j.jmaa.2017.03.055.

\bibitem[CCMN]{ccmn} V. Caselles, A. Chambolle, S. Moll, M. Novaga, \textit{A characterization of convex calibrable sets in $\R^{N}$ with respect to anisotropic norms}, Ann. I. H. Poincar\'e 25, 803-832 (2008).

\bibitem[Ch]{ch} J.Cheeger, \textit{A lower bound for the smalles eigenvalue of the Laplacian}, in:
Problems in Analysis, A Symposium in Honor of Salomon Bochner, Ed.: R.C.Gunning, Princeton Univ. Press (1970) pp. 195-199.

\bibitem[CFV]{cfv} Cozzi M., Farina A., Valdinoci E., \textit{Gradient Bounds and Rigidity Results for Singular, Degenerate, Anisotropic Partial Differential Equations}, Comm. Math. Phys., 189-214 (2014).

\bibitem[CM]{cm07}
Crasta G., Malusa A.,
\textit{The distance function from the boundary in a Minkowski space}.
{Trans. Amer. Math. Soc.}, 359(12):5725--5759 (2007).

\bibitem[CT]{ct} Cuesta M., Tak\'a\v c P., \textit{A strong comparison principle for positive solutions of degenerate elliptic equations}. Diff. Int. Eq. 13 (2000): 721-746.

\bibitem[DG1]{dgmana}  Della Pietra F.,  Gavitone N., \textit{Sharp bounds for the first eigenvalue and the torsional rigidity related to some anisotropic operators}, Math. Nachr. 287, 194-209 (2014).

\bibitem[DG2]{dgmaan} Della Pietra F.,  Gavitone N., \textit{Symmetrization with respect to the anisotropic perimeter and applications}. Math. Ann. 363, 953-971 (2015).

\bibitem[DGG]{dpgg} Della Pietra F., Gavitone N., Guarino Lo Bianco S., On functionals involving the torsional rigidity related to some classes of nonlinear operators, preprint, 2017.

%\bibitem[DG2]{dpgmaan} Della Pietra F., Gavitone N., Symmetrization with respect to the anisotropic perimeter and applications, Math. Ann. 363, 953-971 (2015).

%\bibitem[DGP]{dgp} Della Pietra F., Gavitone N., Piscitelli G., \textit{A sharp weighted anisotropic Poincar\'e inequality for convex domains}, preprint, 2017.

\bibitem[DGP2]{dgp2} Della Pietra F., Gavitone N., Piscitelli G., On the second Dirichlet eigenvalue of some nonlinear anisotropic elliptic operators, preprint, 2017.

\bibitem[FM]{fomu}
I.~Fonseca and S.~M\"{u}ller.
\newblock {A uniqueness proof for the {W}ulff theorem}.
\newblock {\em Proc. Roy. Soc. Edinburgh Sect. A}, 119(1-2):125--136, 1991.

%\bibitem[FGL]{fgl} Fragal\`a I., Gazzola F., Lamboley J., \textit{Sharp Bounds for the $p$-Torsion of Convex Planar Domains}, Geometric Properties for Parabolic and Elliptic PDE's
%Springer INdAM Series Volume 2, 97-115 (2013).

\bibitem[H]{he} J.Hersch \textit{Sur la fr\'equence fondamentale d'une membrane vibrante: \'evaluations par d\'efaut et principe de maximum}, Z.Angew. Math. Phys. 11 (1960), 387-413.

\bibitem[HLP]{hlp} Henrot A., Lucardesi I., Philippin G., \textit{On two functionals involving the maximum of the torsion function}, preprint.


\bibitem[Kaj]{kaj} Kajikiya, R., A priori estimate for the first eigenvalue of the p-Laplacian. Diff. Int. Eq. 28 (2015), 1011-1028.

\bibitem[KN]{kn08}
B. Kawohl, M. Novaga.
\textit{{The $p$-{L}aplace eigenvalue problem as $p\rightarrow 1$ and
  {C}heeger sets in a {F}insler metric}}.
{J. Convex. Anal.} 15 (2008):623-634.

%\bibitem[L]{l95} P. Lindqvist, \textit{Some remarkable sine and cosine functions}, Ric. Mat. 44, 269-290 (1995).

\bibitem[LU]{ladyz}
Ladyzhenskaya O.~A., Ural'tseva N.~N.
Linear and quasilinear elliptic equations.
 Translated from the Russian by Scripta Technica, Inc. Translation
 editor: Leon Ehrenpreis. Academic Press, New York, 1968.

%\bibitem[M]{m} Makai E., \textit{On the principal frequency of a membrane and the torsional rigidity of a beam}, Studies in mathematical analysis and related topics, 227-231, Stanford Univ. Press, Stanford, Calif. 1962.

%\bibitem[P]{pay} {Payne L.E.}, \textit{Bounds for the maximum Stress in the St-Venant problem}, Indian Journal of Mechanics  and Mathematics, special issue in honor of B. Sen, part 1, 51--59  (1968)

\bibitem[P]{pay}
 L.E. Payne, \textit{Bounds for solutions of a class of quasilinear elliptic boundary value problems in terms of the torsion function}. Proc. R. Soc. Edinb. 88A, 251-265 (1981)

\bibitem[PSt]{ps} {L.E. Payne, I. Stakgold}, \textit{On the mean value of the fundamental mode in the fixed membrane problem}. Collection of articles dedicated to Alexander Weinstein on the occasion of his 75th
birthday. Applicable Anal. { 3}, 295-306 (1973).

\bibitem[Pa]{pa}
E. Parini, \textit{Reverse Cheeger inequality for planar convex sets}, {\em J. Conv. Anal} 24 (2017) 107-122.

\bibitem[Pa2]{pa2}
E. Parini, \textit{An introduction to the Cheeger problem}, {\em Surveys in Mathematics and its Applications} 6 (2011) 9-22.
\bibitem[Po]{po} G. Poliquin, \textit{Bounds on the principal frequency of the p-Laplacian}, Contemp. Math.
630 (2014) 349-366.

\bibitem[Pr]{pr} M.H. Protter, \textit{A lower bound for the fundamental frequency of a convex region}, Proc. AMS 81 (1981), 65-70.

\bibitem[To]{tk84}
Tolksdorf P.,
\textit{Regularity for a more general class of quasilinear elliptic
 equations}, J. Diff. Eq., 51(1):126--150 (1984).
 
 \bibitem[Sa]{santalo} Santal\`o L.A., \textit{Sobre los sistemas completos de desigualdades entre tres elementos de una figura convexa plana}, Math. Notae 17, 82-104 (1961).

\bibitem[S]{sp} Sperb R., \textit{Maximum principles and applications}, Academic Press, 1981.

\bibitem[WX]{wxpac}
G.~Wang and C.~Xia.
\newblock {An optimal anisotropic Poincar\'e inequality for convex domains}.
\newblock {\em Pacific J. Math.}, 258(2):305--326, 2012.

\end{thebibliography}
\end{document}